\documentclass[12pt]{amsart}
\usepackage{amsthm,amsmath,amssymb,latexsym}
\usepackage{hyperref}
\begin{document}
\newtheorem{theorem}{Theorem}[section]
\newtheorem{definition}[theorem]{Definition}
\newtheorem{proposition}[theorem]{Proposition}
\newtheorem{lemma}[theorem]{Lemma}
\newtheorem{remark}{Remark}[section]
\newtheorem{corollary}[theorem]{Corollary}
\newtheorem{question}{Question}
\newtheorem{example}{Examples}[section]
\newtheorem{notation}{Notation}[section]
\newtheorem{claim}{Claim}[theorem]
\newtheorem{fact}[theorem]{Fact}
\newcommand\cl{\begin{claim}}
\newcommand\ecl{\end{claim}}
\newcommand\rem{\begin{remark}\upshape}
\newcommand\erem{\end{remark}}
\newcommand\ex{\begin{example}\upshape}
\newcommand\eex{\end{example}}
\newcommand\nota{\begin{notation}\upshape}
\newcommand\enota{\end{notation}}
\newcommand\dfn{\begin{definition}\upshape}
\newcommand\edfn{\end{definition}}
\newcommand\cor{\begin{corollary}}
\newcommand\ecor{\end{corollary}}
\newcommand\thm{\begin{theorem}}
\newcommand\ethm{\end{theorem}}
\newcommand\prop{\begin{proposition}}
\newcommand\eprop{\end{proposition}}
\newcommand\lem{\begin{lemma}}
\newcommand\elem{\end{lemma}}
\newcommand\fct{\begin{fact}}
\newcommand\efct{\end{fact}}
\providecommand\qed{\hfill$\quad\Box$}
\newcommand\pr{{Proof:\;}}
\newcommand\prc{\par\noindent{\em Proof of Claim: }}
\newcommand\dom{{\text{dom}}}
\newcommand\Pn{{P_n}}
\newcommand\ord{{ord}}
\newcommand\rk{{rk}}
\newcommand\G{{G^{\n}}}
\newcommand\car{{\rm{char}}}
\newcommand\M{{{\mathcal M}}}
\newcommand\N{{{\mathcal N}}}
\newcommand\K{{{\mathcal K}}}
\newcommand\D{{{\mathcal D}}}
\newcommand\A{{{\mathcal A}}}
\newcommand\Ps{{{\mathcal P}}}
\newcommand\E{{\mathcal E}}
\newcommand\X{{{\bold{X}}}}
\newcommand\x{{{\bold{x}}}}
\newcommand\Y{{{\bold{Y}}}}
\renewcommand\O{{{\mathcal O}}}
\renewcommand\L{{{\mathcal L}}}
\newcommand\Th{{\text{Th}}}
\newcommand\IZ{{\mathbb Z}}
\newcommand\IQ{{\mathbb Q}}
\newcommand\IR{{\mathbb R}}
\newcommand\IN{{\mathbb N}}
\newcommand\IC{{\mathbb C}}
\newcommand\F{{\mathbb F}}
\newcommand\Se{{\mathcal S}}
\newcommand\V{{\mathcal V}}
\newcommand\W{{\mathcal W}}
\newcommand\T{{\mathcal T}}
\newcommand{\Ecl}{\mathrm{ecl}}
\newcommand{\Ext}{\mathrm{Ext}}
\newcommand{\Tr}{\mathrm{Tr}}
\newcommand{\Cl}{\mathrm{cl}}
\newcommand\si{{\sigma}}
\newcommand\n{{\nabla}}
\newcommand\C{{\mathcal C}}
\newcommand\Lr{{{\mathcal L}_{\text{rings}}}}
\newcommand\Lrd{{{\mathcal L}^*_{\text{rings}}}}
\newcommand\B{{\mathcal B}}
\newcommand\La{{\mathcal L}}
\def\U{{ \mathfrak U}}
\def\B{{\mathfrak B}}
\def\I{{\mathcal I}}
\def\Ma{{\mathfrak M}}
\let\le=\leqslant
\let\ge=\geqslant
\let\subset=\subseteq
\let\supset=\supseteq

\author[Point]{Fran\c coise {Point}$^{(\dagger)}$}
\thanks{$(\dagger)$ Research Director at the "Fonds de la Recherche
  Scientifique (F.R.S.-F.N.R.S.)"}
\address{Department of Mathematics (De Vinci)\\ UMons\\ 20, place du Parc 7000 Mons, Belgium}
\email{point@math.univ-paris-diderot.fr}
\author[Regnault]{Nathalie {Regnault}}
\address{Department of Mathematics (De Vinci)\\ UMons\\ 20, place du Parc 7000 Mons, Belgium}
\email{regnaultnathalie@gmail.com}
\title{Exponential ideals and a Nullstellensatz}
\date{\today}

\begin{abstract} We prove a version of a Nullstellensatz for partial exponential fields $(F,E)$, even though the ring of exponential polynomials $F[\X]^E$ is not a Hilbert ring. We show that under certain natural conditions, one can embed an ideal of $F[\X]^E$ into an exponential ideal. In case the ideal consists of exponential polynomials with one iteration of the exponential function, we show that these conditions can be met. We apply our results to the case of ordered exponential fields.
\end{abstract}
\maketitle
\par MSC classification: primary: 03C60 (secondary: 12L12, 12D15).
 \par Key words:  exponential ideals, Nullstellensatz, partial exponential fields.

\section{Introduction}
\par There is an extensive literature on {\it Nullstellens\"atze} for expansions of fields (ordered, differential, $p$-valued) and several versions of abstract Nullstellens\"atze attempting to encompass those cases in a general framework (see for instance \cite{Mc}, \cite{Weis}). When $K$ is a (pure) algebraically closed field, Hilbert's Nullstellensatz establishes the equivalence between the following two properties: a system of polynomial equations (with coefficients in $K$) has a common solution (in $K$) and the ideal generated by these polynomials is nontrivial in the polynomial ring $K[\X]$. 
\par It is well-known that this equivalence no longer hold in, for instance, the field of complex numbers endowed with the exponential function (see Remark \ref{terzo}).
Nevertheless in this note we will give a version of a Nullstellensatz for exponential fields, namely fields $(K, E)$ endowed with a partially defined  exponential function $E$. We start by recalling the notion of $E$-algebraic closure, first defined by A. Macintyre, then used by A. Wilkie in an o-minimal setting and then investigated by J. Kirby in general. Then we show how to construct maximal (or prime) ideals which are also exponential ideals, using the fact that exponential polynomial rings are union of group rings, pointing out that there are prime ideals that are not an intersection of maximal ideals. 
Finally under a natural condition on the system of exponential polynomials we are dealing with, we prove a Nullstellensatz in this setting (see below for a precise statement).
\par In this line of research, let us point out two former works. In the eighties, K. Manders investigated the notion of exponential ideals \cite{Manders}, and pinpointed several obstructions to develop an analog of the classical Nullstenllensatz. Later in her thesis \cite{T}, G. Terzo pointed out that very few results are known about exponential ideals and that even seemingly basic questions are not yet answered (see Remark \ref{terzopage18}). 
Let us now describe our main result.

\

Let $(F, E)$ be an exponential field, namely a field endowed with a morphism $E$ from its additive group $(F,+,0)$ to its multiplicative group $(F^*,.,1)$ with $E(0)=1$.
\par Using the construction of the free $E$-ring $F[\X]^E$ of exponential polynomials over $F$ on $\bold{X}$ as an increasing union of group rings, namely letting $F_0:=F[\bold{X}]$ and $F\X]^E=\bigcup_{\ell\geq 0} F_\ell$ (see section \ref{free} and \cite{D}), we show the following. First let us state a corollary of our main theorem, where we restrict ourselves to elements of $F_{1}$, namely exponential polynomials with at most one iteration of the exponential function.
\smallskip
\par {\bf Corollary} (Later Corollary 3.11) Let $f_1,\ldots, f_{m}, g\in F_1$ and let $I$ be the ideal of $F_{1}$ generated by $f_{1},\ldots, f_{m}$.
Assume that $g$ vanishes at each common zero of $f_{1},\ldots, f_{m}$ in any partial exponential field containing $(F,E)$. Then $g$ belongs to $J(I)$, the Jacobson radical of $I$.
\smallskip
\par In order to consider the general case of exponential polynomials in $F[\X]^E$, we introduce on an ideal $I$ of $F_\ell$, the following (expected) $E$-compatibility condition at level $\ell$ :
\[
(\Ext)_\ell:\;\;\;\;\forall\;\;h\in F_{\ell-1}\cap I,\;\; E(h)-1\in I. 
\]
\par {\bf Theorem}  (Later Theorem 3.10) Let $f_{1},\ldots, f_{m}, g\in F[\bold{X}]^E$ and let $\ell$ be minimal such that $f_{1},\ldots, f_{m}, g\in F_{\ell}$.  Suppose that there is a maximal ideal $M$ of $F_\ell$ containing $f_{1},\ldots, f_{m}$ and satisfying the $E$-compatibility condition $(\Ext)_\ell$.

 Assume that $g$ vanishes at each common zero of $f_{1},\ldots, f_{m}$ in any partial exponential field containing $(F,E)$. Then $g$ belongs to $M$.
\medskip

\par The paper is organized as follows. In section 2, we recall the construction of free exponential rings (and the complexity function they can be endowed with), the definition of exponential ideals and of the $E$-algebraic closure ($\Ecl$). 
We note that free exponential rings are not Hilbert rings, namely that there are prime ideals which are not an intersection of maximal ideals. 
\par In section 3, we prove our main result (Theorem \ref{Nul}) by constructing step by step maximal (respectively prime) exponential ideals, containing a given proper ideal under the above $E$-compatibility condition $(\Ext)_\ell$, at some appropriate finite level $\ell$. Finally, we apply our techniques when the field of coefficients is ordered, considering real exponential ideals.

\section{Preliminaries}
Throughout, all our rings $R$ will be commutative rings of characteristic $0$ with identity $1$. Let $\IN^*:=\IN\setminus\{0\}$, $R^*:=R\setminus\{0\}$.
We will work in the language of rings $\L_{rings}:=\{+,.,-,0,1\}$ augmented by a unary function $E$; let $\L_{E}:=\L_{rings}\cup\{E\}$. 
\dfn {\rm \cite{D}} An $E$-ring $(R,E)$ is a ring $R$ equipped with a map\\ \mbox{$E: (R,+,0)\to (R^*,\cdot,1)$} satisfying $E(0)=1$ and $\forall x\forall y\;(E(x+y)=E(x).E(y))$. 
\par An $E$-domain is an $E$-ring with no zero-divisors; an $E$-field is an $E$-domain which is a field.
\edfn
\par We will need to work with partial $E$-rings, namely rings where the exponential function is only partially defined and so the corresponding language contains a unary predicate for the domain of the exponential function.  
\dfn  (See \cite[Definition 2.2]{K}).
A partial $E$-ring is a two-sorted structure $(R,A(R),+_{R},\cdot,+_{A},E)$, where $(R,+_{R},\cdot,0,1)$ is a domain, $(A(R),+_{A},0)$ is an abelian monoid and \mbox{$E:(A(R),+_{A},0)\to (R,\cdot,1)$} is a morphism of monoids. Further we have an injective homomorphism of abelian monoids from $(A(R),+_{A})$ to $(R,+_{R})$; we identify $A(R)$ its image in $R$ and write just one operation $+$. We get the corresponding notions of partial $E$-domains and partial $E$-fields.
\edfn
\par Let $R$ be a partial $E$-ring and $A(R)$ a group, then $E(A(R))$ is a subset of the set of invertible elements of $R$.
\par  The above definition of partial $E$-rings differs from \cite[Definition 2.2]{K}. Indeed, there one requires in addition that $R$ is a $\IQ$-algebra, $A(R)$ is a $\IQ$-vector space and as such endowed with scalar multiplication $\cdot q$ for each $q\in \IQ$.
\par To simplify notations and in order to treat the cases of partial $E$-rings and $E$-rings simultaneously, we will denote the partial $E$-ring $(R,A(R),E)$ simply by $(R,E)$.

\ex We recall below classical examples of exponential (partial) $E$-rings and $E$-fields.

\begin{enumerate}
\item Let $\IR$ be the field of real numbers (respectively $\IC$ the field of complexes) endowed with the exponential function $exp(x):=\sum_{n\geq 0} \frac{x^n}{n!}$.
\item Let $\IZ_p$ be the ring of $p$-adic integers endowed with the exponential function $x \mapsto exp(px)$, $p>2$. 
\item Let $\IZ$ be the ring of integers endowed with the exponential function $x\mapsto 2^x$ only defined on the positive integers. 
\item Let $\IQ_p$ be the field of $p$-adic numbers (respectively $\IC_p$ the completion of the algebraic closure of $\IQ_{p}$) endowed with the exponential function $exp(px)$ restricted to $\IZ_p$ (respectively to $\O_p$ the valuation ring of $\IC_p$), $p>2$. 
\end{enumerate}
\par Let $(F,E)$ be a partial $E$-field. Consider the field of Laurent series $F((t))$. Write $r\in F[[t]]$ as $r_0+r_1$ where $r_0\in F$ and $r_1\in t F[[t]]$. For $r_0\in A(F)$, extend $E$ on $F[[t]]$ as follows: $E(r_0+r_1):=E(r_0) exp(r_1)$. By Neumann's Lemma, the series $exp(r_1)\in F((t))$ \cite[chapter 8, section 5, Lemma]{F}. So  $F((t))$ can be endowed with the structure of a partial $E$-field.
\par More generally let $G$ be a totally ordered abelian group and consider the Hahn field $F((G))$. Let $\M_{v}$ denote the maximal ideal of $F((G))$. Then $F((G))$ can be endowed with the structure of a partial $E$-field by defining $E$ on the elements $r$ of the valuation ring which can be decomposed as  $r_0+r_1$ with $r_0\in A(F)$ and $r_1\in \M_v$. Then $exp(r_1)\in F((G))$ by Neumann's lemma cited above and $E(r_0+r_1):=E(r_0) exp(r_1)$, for $r_0\in A(F)$.
\eex 
 \subsection{Free exponential rings}\label{free}
\par For the convenience of the reader not familiar with the construction of the ring of exponential polynomials, we recall its construction below (see for instance \cite{D}, \cite{M}). 

We denote by $\IZ[\X]^E$ the construction on $\X:=X_{1},\ldots,X_{n}$ and by $R[\X]^E$ the construction on the $E$-ring $(R,E)$. 
\par The elements of these rings are called {\it $E$-polynomials} in the indeterminates $\X$.
\par The ring $R[\X]^E$ is constructed by stages as follows:
let $R_{-1}:=R$, $R_{0}:=R[\X]$ and $A_{0}$ the ideal generated by $\X$ in $R[\X]$. We have $R_{0}=R\oplus A_{0}$.
For $k\geq 0$, 
let $t^{A_{k}}$ be a multiplicative copy of the additive group $A_{k}$.  
\par Then, set $R_{k+1}:=R_{k}[t^{A_{k}}]$ and let $A_{k+1}$ be the free $R_{k}-$submodule generated by $t^{a}$ with $a\in A_{k}-\{0\}$. 
Then $R_{k+1}=R_{k}\oplus A_{k+1}$. 
\par By induction on $k\geq 0$, one shows the following isomorphism: $R_{k+1}\cong R_{0}[t^{A_{0}\oplus\cdots\oplus A_{k}}],$ using the fact that $R_{0}[t^{A_{0}\oplus\cdots\oplus A_{k}}]\cong R_{0}[t^{A_{0}\oplus\cdots\oplus A_{k-1}}][t^{A_{k}}]$ \cite[Lemma 2]{M}.
\par We define the map $E_{-1}: R_{-1}\to R_0$ as the map $E$ on $R$ composed by the embedding of $R_{-1}$ into $R_{0}$.
Then for $k\geq 0$, we define the map $E_{k}:R_{k}\to R_{k+1}$ as follows: $E_{k}(r'+a)=E_{k-1}(r').t^a$, where $r'\in R_{k-1}$ and $a\in A_{k}$. Since $R_{k-1}$ and $A_k$ are in direct summand this is well-defined. Moreover for $k\leq \ell$, $E_{\ell}$ extends $E_k$.
\par Finally let $R[\X]^E:=\bigcup_{k\geq 0} R_{k}$ and extend $E$ on $R[\X]^E$ by defining  for $f\in R_k$, $E(f):=E_k(f)$.

\medskip
\par Using the construction of $R[\bold{X}]^E$ as an increasing union of group rings, one can define on the elements of $R[\X]^E$ an analogue of the degree function for ordinary polynomials which measures the complexity of the elements; it will take its values in the class $On$ of ordinals and was described for instance in \cite[1.9]{D} or in \cite[1.8]{M}.
\medskip
\par For $p\in R[\X]$, let us denote by $totdeg_{\bold{X}}(p)$ the total degree of $p$, namely 
the maximum of $\{\sum_{j=1}^n i_{j}\colon$ for each monomial $X_{1}^{i_{1}}\ldots X_{n}^{i_{n}}$ occurring (nontrivially) in $p\}$.

\par Then one defines a height function $h$ (with values in $\IN$) which detects at which stage of the construction the (non-zero) element is introduced.
\par Let $p(\X)\in R[\X]^E$, then $h(p(\X))=k$, if $p\in R_{k}\setminus R_{k-1}$, $k>0$ and $h(p(\X))=0$ if $p\in R[\X]$. 
Using the freeness of the construction, one defines a function $rk$
\[
rk: R[\X]^E\to \IN:
\]
\par If $p=0$, set $rk(p):=0$, 
 \par if $p\in R[\X]\setminus\{0\}$, set $rk(p):=totdeg_{X}(p)+1$ and 
 \par if $p\in R_{k}\setminus R_{k-1}$, $k>0$, let $p=\sum_{i=1}^d r_{i}.E(a_{i})$, where $r_{i}\in R_{k-1}$, $a_{i}\in A_{k-1}\setminus\{0\}$. Set
$rk(p):=d$.
\par Finally, one defines the complexity function $ord$
\[
ord: R[\X]^E\to On
\]  as follows. For $k\geq 0$, write $p\in R_k$ as $p=\sum_{i=0}^k p_{i}$ with $p_0\in R_0$, $p_{i}\in A_{i}$, $1\leq i\leq k$. Define $ord(p):=\sum_{i=0}^k \omega^i.rk(p_{i})$.
\rem \label{order} \cite[Lemma 1.10]{D} Let $p\in R_k\setminus R_{k-1}$, $k\geq 1$ and assume that $p=\sum_{i=1}^k p_{i}$ with $p_{i}\in A_{i}$, $1\leq i\leq k$. Then there is $q\in R[\X]^E$ such that $ord(E(q).p)<ord(p)$.
\erem
\medskip
\par Note that if we had started with a partial $E$-ring $(R,E)$, then $R[\bold{X}]^E$ can be endowed with the structure of a partial $E$-ring. Indeed, let $f\in R_k$, written as $f_0+f_1$, with $f_0\in A(R)$ and $f_1\in A_0\oplus A_1\oplus\cdots \oplus A_{k}$, then define $E(f)$ as $E(f_0).t^{f_1}\in R_{k+1}$, $k\geq -1$. 
\medskip
\par Finally let us recall the definition of the closure operator $\Ecl$. 
 
\nota Given $f_1,\ldots,f_n\in R[\X]^E$, $\bar f:=(f_1,\ldots,f_n)$, we will denote by  $J_{\bar f}(\X)$, the 
Jacobian matrix  $\left( \begin{array}{ccc}
     \partial_{X_1} f_1 & \cdots & \partial_{X_n} f_1 \\
      \vdots & \ddots & \vdots\\
      \partial_{X_1} f_n & \cdots & \partial_{X_n} f_n  
      \end{array} \right)$
and by $\det(J_{\bar f}(\X))$ its determinant; note that $\det(J_{\bar f}(\X))$ is an $E$-polynomial.
\enota
\dfn \label{ecl} \cite[Definition 3.1]{K} Let $(B,E)\subset (R,E)$ be partial E-domains. A {\it Khovanskii system over $B$} is a quantifier-free $\L_E(B)$-formula in $\x:=(x_1,\ldots,x_n)$ of the form 
$$H_{\bar f}(\x):=\bigwedge_{i=1}^n f_i(\x)=0\;\wedge\; \det(J_{\bar f}(\x))\neq 0,$$ for some $f_1,\ldots,f_n\in B[\X]^E.$
\par Let $a\in R$. Then $a\in \Ecl^R(B)$ if  $H_{\bar f}(a_1,\ldots,a_n)$ holds for some $a_2,\ldots,a_n\in R$ with $a=a_1$, where $H_{\bar f}$ is a Khovanskii system, $f_1,\ldots,f_n\in B[\X]^E$ (assuming that $a_i\in A(R)$, $1\leq i\leq n$, if necessary for the $f_i's$ to be defined).
\edfn

\par The operator $\Ecl$ was introduced in \cite{M} and then used by A. Wilkie in his proof of the model-completeness of the field of reals with the exponential function \cite{W}. 
\par Then in a general algebraic framework, J. Kirby showed that $\Ecl$ coincides with another closure operator defined using $E$-derivations \cite[Propositions 4.7, 7.1]{K} and used that property in order to prove that $\Ecl$ satisfies the exchange property \cite[Theorem 1.1]{K}. So the operator $\Ecl$ induces a pregeometry on subsets of $R$ and we get a dimension, that we denote by $\dim$.
\subsection{Hilbert rings} \label{Hilbert} $\;$

\par The Jacobson radical $J(R)$ of a ring $R$ is by definition the intersection of all maximal ideals of $R$. It is a definable subset of $R$, namely $J(R)=\{u\in R:\;\forall z\exists y\;(1+u.z).y=1\}$. Given an ideal $I$ of $R$, we denote by $J(I)$ the intersection of all maximal ideals that contain $I$, so this is equal to $\{u\in R:\;\forall z\exists y\;(1+u.z).y-1\in I\}.$ We denote by $rad(I)$ the intersection of all prime ideals of $R$ that contain $I$ and one shows that $rad(I)=\{u\in R:\; \exists n\in \IN\;u^n\in I\}$.
\par A Hilbert ring (also called Jacobson ring) is a ring $R$ where any prime ideal is the intersection of maximal ideals. Therefore for any ideal $I$, $rad(I)=J(I)$.
\par The terminology of {\it Hilbert ring} comes from the fact that in a Hilbert ring $R$, Hilbert Nullstellensatz holds, namely if a polynomial vanishes at every zero of an ideal $I$, then $f\in rad(I)$.
\par Using a result of J. Krempa and J. Okninski, one can observe that the ring of exponential polynomials is not an Hilbert ring (see Corollary \ref{notH} below). But let us quickly recall in the case of polynomial rings a way to show Hilbert Nullstellensatz \cite{Gol}.

Given a maximal ideal $M$ of $K[X]$, one first shows that the field $K[X]/M$ is algebraic over $K$. Then using that $K[\X]$ is a Hilbert ring, one proves by induction on $n\geq 1$, that $K[\X]/M$ is algebraic over $K$, for $M$ a maximal ideal of $K[\X]$ \cite[Corollary to Theorem 5]{Gol}. 
The following properties are used:
\par (0) If $K$ is a field, then $K[\X]$ is a Hilbert ring \cite[Corollary to Theorem 3]{Gol}.
\par (1) If $R$ is a Hilbert ring, then for any ideal $I$ of $R$, then $R/I$ is also a Hilbert ring \cite[Theorem 1]{Gol}.
\par (2) $R$ is a Hilbert ring if and only if every maximal ideal of the polynomial ring $R[X]$ contracts to a maximal ideal of $R$ \cite[Theorem 5]{Gol}. 
\medskip
\par Let $(K,E)$ be now an exponential field. We will show that $K[\X]^E$ is not an Hilbert ring, and in the next section that a weak form of the algebraicity property mentioned above, holds in the exponential setting (see Lemma \ref{onevariable} and Remark \ref{nvariable}).

\medskip
\par In the following, $G$ denotes a torsion-free abelian group; the {\it torsion-free rank} of $G$ is the dimension of the $\IQ$-vector space $G\otimes \IQ$.

\par Following a result of Krull for polynomial rings, one can show \cite[Proposition 1]{KO}, that if $F$ is a field and $G$ a group of torsion-free rank $\alpha\geq \omega$, then the group ring $F [G]$ is a Hilbert ring if and only if $\vert F\vert>\alpha$. When $F$ is not a field, a necessary and sufficient condition was obtained by Krempa and Okninski \cite[Theorem 4]{KO}. 
Let us state below part of their result.
\thm \cite[Theorem 4]{KO} Let $G$ be an abelian group of infinite torsion-free rank $\alpha$ and let $R$ be a ring. If $R[G]$ is a Hilbert ring, then all homomorphic images of $R$ have cardinality exceeding $\alpha$.
\ethm
\cor \label{notH} The $E$-ring $K[\X]^E$ is not a Hilbert ring.
\ecor
\pr Let $K$ be an $E$-field and consider the group rings $K_{\ell}$, $\ell\geq 1$. Recall that $K_{\ell}\cong K_{0}[t^{A_{0}\oplus\cdots\oplus A_{\ell}}]$ (see section \ref{free}) and $K[\X]^E\cong K_{0}[t^{\bigoplus_{\ell\geq 0} A_{\ell}}]$. Let $\alpha_{\ell}$ be the torsion-free rank of the (multiplicative) group $t^{A_{0}\oplus\cdots\oplus A_{\ell}}$. Then all the homomorphic images of $K_{0}$ do not have cardinality $>\alpha_{\ell}$ and so the first condition of Theorem 4 in \cite{KO} fails and so $K_{\ell}$ is not a Hilbert ring. A similar reasoning applies for $K[\X]^E$.
\qed
\medskip
\par To conclude this section let us recall the following observation of  P. D'Aquino,  A. Fornasiero, G. Terzo \cite{DFT}. Let $(\IC,exp)$ be the field of complex numbers endowed with the classical exponential function $exp$. 
\rem \label{terzo} Let $c\in \IC\setminus exp(2\pi\IZ)$. Consider $f(X)=exp(X)-c$ and $g(X)=exp(iX)-1$. Let $I$ be the $E$-ideal in $\IC[X]^E$ generated by $f$ and $g$. Then even though $I$ is a nontrivial ideal of  $\IC[X]^E$, $f$ and $g$ have no common zero in $\IC$ (or in any pseudo-exponential field (as defined by B. Zilber) containing $\IC$) \cite{DFT}. 
\erem

\subsection{Exponential ideals} 
\dfn Let $(R,E)$ be a partial $E$-ring.
\par An $E$-ideal $J$ of $(R,E)$ is an ideal such that for any $h\in J\cap A(R)$, $E(h)-1\in J$. 
\par An $E$-ideal $J$ is prime if $R/J$ is a domain.
\par Given a subset $A$ of $(R,E)$, we denote by $\langle A\rangle$ the ideal generated by $A$ and by $\langle A\rangle^E$ the smallest $E$-ideal containing $A$.
\edfn
An example of an $E$-ideal is the set of $E$-polynomials in $R[\bold{X}]^E$ which vanish on a subset of $R^n$. When $R$ is a field, such ideal is a prime $E$-ideal.

\medskip 
\par Let $J$ be an $E$-ideal of $(R,E)$. Then the ring $R/J$ can be endowed with an exponential function $E_{J}:R/J\rightarrow R/J$ sending $g+J$ to $E(g)+J$, $g\in A(R)$. This is well-defined since for $h,\;h'\in J\cap A(R)$, $E(h)-E(h')=(E(h-h')-1).E(h')\in J$.
\medskip
\par The following lemma is well-known. For convenience of the reader we give a proof below.
\lem \label{der} Let $(F_0,E)\subset (F,E)$ be two partial $E$-fields. 
 Assume that $c\in F$ is such that there is an $E$-polynomial  
$p(X)\in F_0[X]^E$ such that $p(c)=0$ ($X$ a single variable). Then $c\in \Ecl(F_{0})$.
\elem
\pr Recall that in section \ref{free}, we have defined the ring of exponential polynomials $F_{0}[X]^E$ by induction setting $R_{0}=F_0[X]=F_{0}\oplus A_{0}$ and $R_{i}:=R_{i-1}[t^{A_{i-1}}]=R_{i-1}\oplus A_{i}$, $i>0$. 
\par W.l.o.g. we may assume that $ord(p)$ is minimal such that $p(c)=0$. Write $p$ as: $p=p_{0}+\sum_{i=1}^k p_{i}$, with $p_{0}\in F_{0}[X]$ and $p_{i}\in A_{i}$, $i>0$.
 \par First let us note that $p_{0}$ is non trivial. Suppose otherwise that $p_{0}= 0$, then by Lemma 1.10 in \cite{D} (see also Remark \ref{order}), there exists $q\in F_{0}[X]^E$ such that $ord(E(q) p)<ord(p)$. Since if $p(c)=0$, then $E(q(c)) p(c)=0$, we get a contradiction.  We may assume further that $p_{0}$ is a monic polynomial.
\par Denote by $\partial_{X} p$ the partial derivative of $p$ with respect to the variable $X$.
Since $p_0\neq 0$, $ord(\partial_X p)< ord(p)$ \cite[Lemma 3.3]{D}. 
So, by choice of $p$, $\partial_X p(c)\neq 0$.  So $u\in \Ecl^{F_{0}}(F)$.
\qed

\medskip

\lem \label{onevariable} Let $\Ps$ be a prime $E$-ideal of $K[X]^E$ and let  $F$ be the fraction field of $K[X]^E/\Ps$.
Then $F$ is included in $\Ecl^{F}(K).$
\elem
\pr Let $u:=X+\Ps\in K[X]^E/\Ps$. Then let $p(X)\in \Ps$ an element of minimal order, so we have $p(u)=0$ and w.l.o.g. we may assume that $p$ is monic.
Consider the partial $E$-field extension $F$ of $K$. Then by Lemma above, $u\in \Ecl^{F}(K)$. Since, $\Ecl^{F}(K)$ is a partial $E$-subfield containing $u$ \cite[Lemma 3.3]{K}, it contains $F$.
\qed 

\rem \label{nvariable} More generally, letting $\Ps$ be a prime $E$-ideal of $K[\X]^E$, then $\dim(K[\X]^E/\Ps)<n$ \cite[Corollary 3.8]{PR}. The proof of that last result uses a theorem of Macintyre \cite[Theorem 15 and Corollary]{M}. 
\erem
\subsection{Group rings and augmentation ideals} Now let $S_{0}$ be any ring of characteristic $0$ (not necessarily a polynomial ring) and let $G$ be a torsion-free abelian group. 
\dfn \label{aug-def} We consider the group ring $S_1:=S_{0}[G]$ 
and we define a map $\phi^a$ from $S_{1}\rightarrow S_{0}: \sum_{i} r_{i}.g_{i}\rightarrow \sum_{i} r_{i},$ with $g_{i}\in G$, $r_{i}\in S_{0}$. The kernel of the map $\phi^a$ is called the {\it augmentation ideal} of $S_1$. 
\edfn
Recall that the augmentation ideal is generated by elements of the form $g-1$, $g\in G$ (write $\sum_{i} r_{i}.g_{i}$ as $\sum_{i} r_{i}.(g_{i}-1)+\sum_{i} r_{i}$). 

\section{Exponential ideals and a Nullstellensatz}
\subsection{Embedding an $E$-ideal into a maximal ideal that is an $E$-ideal}

\nota \label{aug-nota} Let $B_{0}$ be a ring of characteristic $0$, let $G$ be a torsion-free abelian group and 
let $B_1$ be the group ring $B_{0}[G]$. Let  $I$ be an ideal of $B_{0}$.
Then compose the augmentation map $\phi^a: B_{1}\rightarrow B_{0}$ with the map sending $B_{0}$ to $B_{0}/I$. Denote the composition of these two maps: $\phi^a_{I}$ and denote by $I_1$ the kernel of $\phi^a_I$ in $B_1$.
\enota
\lem  \label{aug} Let $B_{1}=B_0[G]$ be the group ring $B_{0}[G]$. Let $I$ be an ideal of $B_0$ and let $\phi^a$, $\phi^a_I$ and $I_1$ as in Notation \ref{aug-nota}.
Then,
\begin{enumerate} 
\item $I_1\cap B_0=I$,
\item if $I$ prime, then $I_1$ prime,
\item if $I$ maximal, then $I_1$ maximal.
\end{enumerate}\qed
\elem

\medskip

\par Now we will place ourselves in the group rings $R_n$, $n\geq 1$, defined in $2.1$, assuming that $(R,E)$ is an $E$-field and keeping the same notations as in subsection $2.1$. (In particular all ideals of $R_n$ are, as additive groups, $\IQ$-vector spaces.) Given an ideal of $R_n$, we want to find a natural condition under which we can extend it to an exponential ideal of $R[\X]^E$. We do it by steps, using the above lemma. In order to extend a proper ideal $I_n$ of $R_n$ to a proper ideal of $R_{n+1}$, we will modify the augmentation ideal map ''along $I_{n}$''. 
We will require on $I_n$ the following property:
\[(\Ext)_n\;\;
u\in I_n\cap R_{n-1}\Rightarrow E(u)-1\in I_n.
\]  
Note that if $I_n$ embeds in an $E$-ideal $I$ with the property that $I \cap R_n=I_n$, then $I_n$ has the property $(\Ext)_n$.
\nota\label{summand} 
Recall that $R_n=R_{n-1}\oplus A_n$ and denote by $\pi_{A_n}$ (respectively by $\pi_{R_{n-1}}$) the projection on $A_n$ (respectively on $R_{n-1}$). Recall also that $I_n\cap R_{n-1}$ is a divisible abelian subgroup of $I_{n}$.  Therefore $I_{n-1}:=I_n\cap R_{n-1}$ has a direct summand $\tilde I_n\subset I_{n}$ in $I_n$,
namely
\[ I_n=\tilde I_n\oplus I_{n-1}
\]
Note that $\pi_{A_n}$ is injective on  $\tilde I_n$:
let $u, v\in \tilde I_n$, and write $u=u_0+u_1$, $v=v_0+v_1$, with $u_0,v_0\in R_{n-1}$ and $u_1,v_1\in A_n$. Suppose $u_1=v_1$, then $u-v=u_0-v_0\in R_{n-1}\cap I_{n}\cap \tilde I_n=I_{n-1}\cap \tilde I_{n}=\{0\}$, consequently $u=v$.
\par Let 
\[ A_n=\pi_{A_n}(\tilde I_{n})\oplus \tilde A_n.
\]
\enota
In the statement of the following lemma, we will use Notation \ref{summand}. 
\begin{lemma}\label{aug+}  Let $I_{n}$ be a proper ideal of $R_{n}$ and let $u\in R_n[t^{A_n}]$, then $u$ can be rewritten in a unique way as 
\[\sum_{i} r_iE(u_i)
\]
where $r_i\in R_n$, $u_i\in \tilde I_n\oplus \tilde A_n$, and for $i\neq j$, $u_i\neq u_j$. In other words, the group ring $R_n[t^{A_n}]$ is isomorphic to the group ring $R_{n}[E(\tilde I_n\oplus \tilde A_n)]$.
\end{lemma}
\begin{proof} Let $u=\sum_{i} r_i.t^{a_i}\in R_n[t^{A_n}]$, where $r_i\in R_n$ and $a_i\in A_n$. Decompose $a_i$ as 
\[
a_{i0}+a_{i1}
\]
with $a_{i0}\in \pi_{A_n}(\tilde I_n)$ and $a_{i1}\in \tilde A_n$. Since $\pi_{A_{n}}$ is injective on $\tilde I_{n}$, there exists a unique $f_i\in \tilde I_n$ such that $a_{i0}=\pi_{A_{n}}(f_i)$ and so $f_i=\pi_{R_{n-1}}(f_{i})+a_{i0}$. Set $f_{i0}:=\pi_{R_{n-1}}(f_{i})$.

 We have $E(f_i)=E(f_{i0}).t^{a_{i0}}$ (since $E(a_{i0})$ as been defined as $t^{a_{i0}}$ (see section \ref{free})), and $t^{a_i}=t^{a_{i0}}.t^{a_{i1}}=E(-f_{i0}).E(f_i).t^{a_{i1}}$. Observe that both $E(-f_{i0})\in R_n$, $E(f_{i0})\in R_{n}$ and $E(f_{i})\in R_{n}[t^{A_{n}}]$. Moreover, since $a_{i1}\in A_{n}$, $t^{a_{i1}}=E(a_{i1})$.
 
 So we may re-write $u$ as 
 \[\sum_{i} (r_i E(-f_{i0})) E(f_i) t^{a_{i1}}=\sum_{i} (r_i E(-f_{i0})) E(f_i+a_{i1})\]
  with $r_i E(-f_{i0})\in R_{n}$ and $f_i+a_{i1}\in \tilde I_n\oplus \tilde A_n$. Such expression is unique since the projection $\pi_{A_{n}}$ on  $\tilde I_n$ is injective: for $f\neq g\in \tilde I_n$, $\pi_{A_{n}}(f)\neq \pi_{A_{n}}(g)\in A_n$. So if $u_{i}\neq u_{j}\in \tilde I_n\oplus \tilde A_n$, then $\pi_{A_{n}}(u_{i})\neq \pi_{A_{n}}(u_{j})$.
\end{proof}

\prop \label{E} Let $n\in \IN$ and $I_n$ be a proper ideal of $R_{n}$ with the property $(\Ext)_{n}$. Then, $I_n$ embeds in a (proper) ideal $I_{n+1}$ of $R_{n+1}$ such that 
\begin{align*}
(\Ext)_{n+1}\;\; \;&E(f)-1\in I_{n+1} \;{\rm for\; any\;} f\in I_n,\;  \;{\rm and}\\
(\Tr)_{n+1}\;\;\;&I_{n+1}\cap R_n=I_n.
 \end{align*}
Moreover if $I_{n}$ is prime (respectively maximal), then $I_{n+1}$ is prime (respectively maximal).
\eprop 
\pr Let $I_n$ be a proper ideal of $R_{n}$, $n\in \IN$ (in particular $I_{n}\cap R_{-1}=\{0\}$). By the preceding lemma, any $u\in R_{n+1}$ 
can be rewritten in a unique way as $\sum_{i=1}^{\ell} r_i E(u_i)$, where $r_i\in R_{n},\;u_i\in \tilde I_n\oplus \tilde A_n$, and the $u_i$'s are distinct. 
So the map $\phi$ sending $\sum_{i=1}^{\ell} r_i E(u_i)$ to $\sum_{i=1}^{\ell} r_i\in R_n$ is well-defined and it is a ring morphism from $R_n[E(\tilde I_n\oplus \tilde A_n)]$ to $R_n$. 

The kernel $ker(\phi)$ contains $\{E(f)-1:\;f\in \tilde I_n\}$. Let $\phi_{I_n}$ be the map sending $\sum_{i=1}^{\ell} r_i E(u_i)$ to $\sum_{i=1}^{\ell} r_i+I_n\in R_n/I_n$. 

Define $I_{n+1}$ as $ker(\phi_{I_n})$. By Lemma \ref{aug}, $I_{n+1}$ is an ideal of $R_{n+1}$ with the property that $ker(\phi_{I_n})\cap R_n=I_n$ $(\Tr)_{n+1}$.
\par It remains to show that $I_{n+1}$ contains $E(f)-1$ for any $f\in I_n$  $(\Ext)_{n+1}$.
 Let $f\in I_n$ and write $f$ as $f_0+f_1$ with $f_0\in I_{n-1}$ and $f_1\in \tilde I_n$. 
 
 Then $E(f)-1=(E(f_1)-1).E(f_0)+(E(f_0)-1)$.  We already observe that $E(f_1)-1\in ker(\phi)$ and by assumption $(\Ext)_{n}$, $E(f_0)-1\in I_n$ for $f_0\in I_{n-1}$.  So, $E(f)-1\in ker(\phi_{I_n})$.
 \par Applying Lemma \ref{aug} with $S_0=R_n$ and $G=E(\tilde I_n\oplus \tilde A_n)$, if $I_{n}$ is prime (respectively maximal), then $I_{n+1}$ is prime (respectively maximal).
 \qed

\cor \label{E-E} Let $I_{n_{0}}$ be a proper ideal of $R_{n_{0}}$, $n_{0}\geq 0$, with the property $(\Ext)_{n_{0}}$. Then, $\langle I_{n_0}\rangle^E$ 
 is such that $\langle I_{n_0}\rangle^{E}\cap R_{n_{0}}=I_{n_{0}}$. Moreover if $I_{n_{0}}$ is prime (respectively maximal), then $\langle I_{n_0}\rangle^{E}$ is prime (respectively maximal).
\ecor
\begin{proof} Proposition \ref{E} allows to construct a proper (respectively prime, maximal) ideal $I_{n_{0}+1}$ of $R_{n_0+1}$ containing $E(I_{n_{0}})-1$ and satisfying $(\Tr_{n_0+1})$. Therefore we may reiterate the construction. Then let $I^E:=\bigcup_{n\geq n_{0}} I_n$. It is an $E$-ideal by construction, and it is proper because  for all $n\geq n_{0}$, $I^E\cap R_{n}=I_n$. If $I_{n_{0}}$ is prime, then each $I_n$ is prime for $n\geq n_0$ (by Proposition \ref{E}). So $\langle I_{n_0}\rangle^{E}$ is prime as the union of a chain of prime ideals. 
If $I_{n_0}$ is maximal, then for $n\geq n_{0}$, each $I_{n}$ is maximal by Proposition \ref{E}. So $\langle I_{n_0}\rangle^{E}$ is maximal as the union of a chain of maximal ideals.
\end{proof}
 \medskip
 \par G. Terzo in her thesis asked several questions on $E$-ideals in $R_0[X]$ \cite{T}. In particular whether the $E$-ideal generated by an irreducible polynomial is a prime ideal. (By irreducible $E$-polynomial we mean an $E$- polynomial which cannot be written as a product of two non invertible elements in $R[X]^E$.)
 \rem \label{terzopage18} Let $p(X)\in R_{0}$ and suppose $p(X)$ is irreducible. Since the ideal generated by $p(X)$ in $R_{0}$ is a maximal ideal, by the above corollary, $\langle p(X)\rangle^E$ is a maximal ideal.
 \erem
 \medskip
 \par A natural question is when an ideal $I\subset R_{n}$ satisfies the condition $(\Ext)_n$. Assume that $(R,E)$ has an extension $(S,E)$ where there is a tuple of elements $\bar \alpha$ with the property that for any $f\in I$, $f(\bar \alpha)=0$. Then consider $I_{\bar \alpha}:=\{g\in R_{n}:\;g(\bar \alpha)=0\}$. By definition $I\subset I_{\bar \alpha}$, $I_{\bar \alpha}$ is a (prime) ideal and for any $f\in I\cap R_{n-1}$, $E(f)-1\in I_{\bar \alpha}$.
 
 \medskip
\par In the following proposition, we will examine the condition $(\Ext)_{1}$ that we put on an ideal $I$ of $R_{1}$ in order to embed it in an $E$-ideal. (This corresponds to the case of $E$-polynomials with only one iteration of $E$.)
In this particular case, we can use the fact that $R_{0}$ is a Noetherian ring.

\prop \label{R_{1}} Let $I$ be a proper ideal of $R_{1}$. Then we can embed $I$ in a proper $E$-ideal of $R[\X]^E$.
\eprop

\pr It suffices to show that we can embed $I$ in an ideal $J$ of $R_1$ such that for any $f\in J\cap R_0$, $E(f)-1\in J$ and use Corollary \ref{E-E}.
Since $I$ is a proper ideal of $R_{1}$, $I\cap R_{-1}=\{0\}$. Set $I_{0}:=R_{0}\cap I$. Using Notation \ref{summand}, we get that $\tilde I_0=I_0$.
Recall that $\tilde A_{0}$ is a direct summand of  $\pi_{A_0}(I_{0})$ in $A_{0}$. By Lemma \ref{aug+}, $R_1\cong R_0[E(I_0\oplus \tilde A_0)]$.
\par Set $R_0':=R_0[E(\tilde A_0)]$, so we get $R_1\cong R_0'[E(I_0)]$.
Consider the ideal $I_{0}':=I\cap R'_0$. (Note that $I_{0}=I_{0}'\cap R_0$, since $R_0'\cap R_0=R_0$.)
Let $u\in R_1$ and write it as $\sum_{j} r_{j}.E(u_j)$, where $r_{j}\in R_{0}',\;u_j\in I_0$, with the $u_j$'s distinct. 
The map $\phi^+$ sending $u$ to $\sum_{j=1}^{\ell} r_{j}\in R_{0}'$ is well-defined. 
Define $J$ as $ker(\phi_{I_0'}^+)$,
 then $J$ is an ideal of $R_{1}$ containing $E(f)-1$ for any $f\in I_{0}$ with the property that $J\cap R_{0}'=I_{0}'$ by Lemma \ref{aug}. (It implies that $J\cap R_{0}=J\cap R_{0}'\cap R_0=I_{0}'\cap R_0=I_{0}$. 
\par  Let $u\in \langle J, I\rangle\cap R_{0}'$; then $u=\sum_{i} u_{i}a_{i}$ with $u_{i}\in J, a_{i}\in I$. Since $u\in R_{0}'$, we have that $\phi^{+}(u)=u$. So $u=\sum_{i} \phi^{+}(u_{i}).\phi^{+}(a_{i})$. But $\phi^{+}(u_{i})\in I_{0}'$ and so since $\phi^{+}(a_{i})\in R_{0}'$ and $I_{0}'$ is an ideal, we get that $u\in I_{0}'$. In particular $\langle J, I\rangle$ is proper.
\par  So we repeat the same procedure replacing $I_0$ by $\langle J, I\rangle\cap R_0$. Since $R_0$ is noetherian, the process will stop. So we get a proper ideal $\tilde J$ containing $I$ with the property that for any $f\in \tilde J\cap R_0$, $E(f)-1\in \tilde J.$ So we may apply Corollary \ref{E-E} and embed $\tilde J$ in an exponential ideal of $R[X]^E$.
 \qed 

\subsection{Rabinowitsch's trick}
Recall that Rabinowitsch's trick corresponds to the introduction of an extra variable, allowing one to deduce the algebraic strong Nullstellensatz from the weak one.
Given $f_1(\X),\ldots, f_m(\X)\in R[\X]$ and another polynomial $g(\X)\in R[\X]$ vanishing on all common zeroes of $f_{1},\ldots, f_{m}$, introducing the new variable $Y$, one gets: $f_1,\ldots, f_m, 1-Yg$ do not have any common zeroes.

By the weak Nullstellensatz, this implies that the ideal generated by these polynomials is not proper. So one expresses (by an equality) that $1$ belongs to the ideal generated by $f_1,\ldots, f_m$ and $ 1-Y.g$. 
Then one substitutes $g^{-1}$ to $Y$ in the equality, and clears denominators. This entails that some power of $g$ belongs to the ideal generated by $f_{1},\ldots, f_{m}$ in $R[\X]$.

To mimick Rabinowitsch's trick in $R[\X]^{E}$, where $(R,E)$ is as previously an $E$-field, we proceed as follows introducing a ''non-$E$'' variable, extending $R[\X]^E$ to $R[\bold{X}]^E\otimes_{R} R[Y]$. This partial $E$-ring is isomorphic to the following chain of partial $E$-rings.
Recall that $R_{-1}:=R$ and $R_{0}:=R[\X]$. 
Denote by $S_{0}:=R[\bold{X},Y]=R_{0}[Y]\cong R_{0}\otimes_{R} R[Y]$.
 Let $S_{1}:=S_{0}[t^{A_{0}}]\cong R_{1}\otimes_{R} R[Y]$, 
and by induction on $n\geq 2$, let $S_{n}=S_{n-1}[t^{A_{n-1}}]\cong R_{n}\otimes_{R} R[Y]$. Let $\Se:=\bigcup_{n\geq 0} S_{n}$. Then $\Se\cong R[\bold{X}]^E\otimes_{R} R[Y]$.
\par Now we consider an ideal $J$ of $R[\bold{X}]^E\otimes_R R[Y]$ 
and we want to extend it  into an $E$-ideal, namely an ideal containing $E(f)-1$, for $f\in J\cap R[\bold{X}]^E$. 
\prop \label{E-Y} Let $J_n$ be a proper ideal of $S_{n}$, let $I_n:=J_{n}\cap R_{n}$. Suppose that $I_n$ satisfies $(\Ext)_n$ as an ideal of $R_n$.
Then $J_n$ embeds into a proper ideal $J_{n+1}$ of $S_{n+1}$ such that $I_{n+1}:=J_{n+1}\cap R_{n+1}$ satisfies $(\Ext)_{n+1}$ and $(\Tr)_{n+1}$.
Moreover if $J_{n}$ is prime (respectively maximal), then $J_{n+1}$ is prime (respectively maximal).
\eprop
\pr By Lemma \ref{aug+}, $R_{n}[t^{A_{n}}]\cong R_{n}[E(\tilde I_{n}\oplus \tilde A_{n})]$ (see Notation \ref{summand}). Since $S_{n+1}=S_{n}[t^{A_{n}}]\cong R_{n}[t^{A_{n}}]\otimes_{R}R[Y]$, we get that $S_{n+1}\cong S_{n}[E(\tilde I_{n}\oplus \tilde A_{n})]$.
 Rewrite $s\in S_{n+1}$ as $\sum_{i=1}^{\ell} s_i.E(u_i):\;s_i\in S_{n},\;u_i\in \tilde I_n\oplus \tilde A_n$, with $u_i$ distinct. Let $\phi$ be the map sending $\sum_{i=1}^{\ell} s_i.E(u_i)$ to $\sum_{i=1}^{\ell} s_i\in S_n$ and let $\phi\restriction R_{n+1}$ be the restriction of the map $\phi$ to $R_{n+1}$. 
 Let $\phi_{J_n}$ be the map sending $\sum_{i=1}^{\ell} s_i.E(u_i)$ to $\sum_{i=1}^{\ell} s_i+J_n\in S_n/J_n$. Define $J_{n+1}$ as the kernel of $\phi_{J_n}$; by Lemma \ref{aug}, it is an ideal of $S_{n+1}$ and $ker(\phi_{J_n})\cap S_n=J_n$. Furthermore $ker(\phi_{J_{n}})$ contains $ker((\phi\restriction R_{n+1})_{I_{n}})$ and so by Proposition \ref{E}, it contains $E(f)-1$ for any $f\in I_n$.
\par Applying Lemma \ref{aug} with $B_0=S_n$ and $G=E(\tilde I_n\oplus \tilde A_n)$, if $J_{n}$ is prime (respectively maximal), then $J_{n+1}$ is prime (respectively maximal).
 \qed
 
\cor \label{E-Y-max} Let $J$ be a proper ideal of $S_{n}$ with $n\geq 0$ chosen minimal such. Assume that $J\cap R_{n}$ satisfies the property $(\Ext)_{n}$ (as an ideal of $R_{n}$).
Then $\langle J\rangle^E$ is a proper $E$-ideal  of $S[\X]^E$ with $\langle J\rangle^{E}\cap S_{n}=J$. Moreover whenever $J$ is prime (respectively maximal) in $S_{n}$, then $\langle J\rangle^{E}$ is prime (respectively maximal) in $S[\X]^E$.
\ecor
\pr We apply Proposition \ref{E-Y}.\qed

\subsection{Nullstellensatz}
\begin{corollary}\label{WN} {\rm (Weak Nullstellensatz)} Let $(R,E)$ be an $E$-field and $f_1,\ldots, f_m\in R[\X]^E$. Let $n\in\mathbb N$ be chosen minimal such that $f_1,\ldots, f_m\in R_n$. Assume the ideal $I_{n}$ generated by $f_1,\ldots, f_m$ is proper and that there is a maximal ideal $M_{n}$ of $R_n$ containing $I_{n}$ with the property $(\Ext)_{n}$ (as an ideal of $R_{n}$).
Then $f_1,\ldots, f_m$ have a common zero in an $E$-field extending $(R,E)$. 
\end{corollary}
\begin{proof} By Proposition \ref{E-Y}, the $E$-ideal $\langle M_{n}\rangle^E$ is a proper maximal $E$-ideal of $R[\X]^E$. The quotient $R[\X]^E/\langle M_{n}\rangle^E$ is an $E$-field in which $(X_{1}+\langle M_{n}\rangle^E,\ldots, X_{n}+\langle M_{n}\rangle^E)$ is a common zero of $f_1,\ldots, f_m$.
\end{proof}

\rem \label{WNrem}$\;$
\begin{enumerate}
\item In the statement of the corollary above we may assume that the $E$-polynomials $f_{1},\ldots,f_{m}$ have non-trivial polynomial parts, namely that for $1\leq i\leq m$, $ord(f_{i})$ is of the form $\omega^n.m_{n}+\ldots+m_{0}$ with $m_{n}.m_{0}\neq 0$, $m_{\ell}\in \IN$, $0\leq \ell\leq n$  (see Remark \ref{order}).
\item If, in the above corollary we simply assume that the ideal $I_{n}$ generated by $f_1,\ldots, f_m$ is proper and contained in a prime ideal $P_{n}$  with the property $(\Ext)_{n}$ (as an ideal of $R_{n}$), then $f_1,\ldots, f_m$ have a common zero in an $E$-domain extending $(R,E)$ (and so in the field of fractions of this $E$-domain). 
\end{enumerate}
\erem

\dfn Let $I$ be an ideal of $R_{\ell}$, $\ell\geq 1$. We define $J_{E}(I)$ as the intersection of all maximal ideals $M$ of $R_{\ell}$ containing $I$ with the property$(\Ext)_{\ell}$.
If there are no such maximal ideals, then we set $J_{E}(I)=R_{\ell}.$
\edfn
\thm \label{Nul} Let $f_{1},\ldots, f_{m}, g\in R[\bold{X}]^E$ and let $\ell$ be minimal such that $f_{1},\ldots, f_{m}, g\in R_{\ell}$.  Let $I$ be the ideal of $R_{\ell}$ generated by $f_{1},\ldots, f_{m}$. Suppose that $J_{E}(I)$ is a proper ideal. 
Assume that $g$ vanishes at each common zero of $f_{1},\ldots, f_{m}$ in any (partial) $E$-field containing $(R,E)$. Then $g$ belongs to $J_{E}(I)$.
\ethm
\pr 
Consider the ring $\Se$ and the element $1-Y.g$, then $1-Y.g\in S_{\ell}$. Since $J_{E}(I)$ is proper, there is a maximal ideal $M$ of $R_{\ell}$ containing $I$ with property $(\Ext)_{\ell}$. Let
 $J$ be an ideal of $S_{\ell}$ containing $M$ and $1-Y.g$. Assume that $J$ is proper, so $J\cap R_\ell=M$ and it embeds into a maximal ideal of $S_\ell$. Then, we embed it in a maximal $E$-ideal $\tilde M^{E}$ of $\Se$, by Proposition \ref{E-Y}.  
Since $\tilde M^E\cap R[\X]^E$ is a prime ideal, the quotient $R[\X]^E/\tilde M^E\cap R[\X]^E$ is an $E$-domain containing $(R,E)$ which embeds in the partial $E$-field $\Se/\tilde M^E$ where $f_{1},\ldots, f_{m}, 1-gY$ have a common zero, a contradiction. 
\par Therefore we have 
\[1=\sum_{i} t_{i}(X,Y).h_{i}(X)+(1-Y.g).r(X,Y),\] with $h_i(X)\in M$, $t_{i}(X,Y), r(X,Y)\in  S_{\ell}$. Note that the elements of $S_{\ell}$, $\ell\geq 1$, are of the form $S_{0}[t^{A_{0}\oplus \ldots \oplus A_{\ell-1}}]\cong R_{\ell}\otimes_{R} R[Y]$. So we may substitute $g^{-1}$ to $Y$ and find a sufficiently big power $g^d$, $d>0$, of $g$ such that by multiplying each $t_{i}(X,g^{-1})$, $r(X,Y)$, we obtain again an element in $S_{\ell}$. Since $t_i(X,g^{-1}).g^{d}\in R_{\ell}$, we get 
\[g^{d}=\sum_{i} (t_i(X,g^{-1}).g^{d}).h_{i}(X)\in R_{\ell}.\]  Therefore $g^d\in M$ and since $M$ is maximal, $g\in M$. Since we can do that for any maximal ideal $M$ containing $I$ with property $(\Ext)_{\ell}$, we get that $g\in J_{E}(I)$.
 \qed

\

\rem The problem with replacing in the theorem above, maximal ideals $M$ with prime ideals $\Ps$ is that $J\cap R_{\ell}$ might be bigger than $\Ps$ and so the condition $(\Ext)_\ell$ might not hold anymore.
\erem

\par Using Proposition \ref{R_{1}}, we may deduce a stronger result in the case when the $E$-polynomials have only one iteration of the exponential function (namely the case $\ell=1$).
\cor Let $f_{1},\ldots, f_{m}, g\in R_1$ and let $I$ be the ideal of $R_{1}$ generated by $f_{1},\ldots, f_{m}$.
Assume that $g$ vanishes at each common zero of $f_{1},\ldots, f_{m}$ in any partial $E$-field containing $(R,E)$. Then $g$ belongs to $J(I)$, the Jacobson radical of $I$.
\ecor

\pr We apply Theorem \ref{Nul} and we may replace $J_E(I)$ by the ordinary Jacobson radical of $I$ since by Proposition \ref{R_{1}}, every maximal ideal of $R_1$ satisfies $(\Ext)_1$.
\qed
\par Finally we ask ourselves the question of when does a system of $E$-polynomials with coeffcients in $F$ have a common zero in an $Ecl$-closure of $R$.
\cor\label{Kh} Let $f_1,\ldots,f_n\in R_k$ and assume that $k\in \IN$ is minimal such. Assume that the ideal $I:=(f_1,\ldots,f_n)$ of $R_k$ generated by $f_1,\ldots,f_n$ is disjoint from $\{det(J_{\bar f})^m:\;m\in \IN^*\}$. Let $\Ps$ be an ideal of $R_k$ containing $I$ and maximal for the property of being disjoint from $\{det(J_{\bar f})^m:\;m\in \IN^*\}$. Suppose $\Ps$ satisfies $(\Ext)_k$. 
Then $f_1,\ldots,f_n$ have a zero $\bar \alpha\in \Ecl^{F}(R)$ where $(F,E)$ in a partial $E$-field extending $(R,E)$.
\ecor
\pr
It is well-known that $\Ps$ is a prime ideal of $R_k$. Since we assumed that it satisfies the hypothesis that for any $f\in \Ps\cap R_{k-1}$, $E(f)-1\in \Ps$, we may apply Corollary \ref{E-E} and get that $\langle \Ps\rangle^E$ is a prime $E$-ideal of $R[\X]^E$ with $\langle \Ps\rangle^E\cap R_k=\Ps$. Then the element $\bar \alpha:=(X_1+\langle \Ps\rangle^E,\ldots,X_n+\langle \Ps\rangle^E)$ satisfies the formula $H_{\bar f}$ in the $E$-domain $R[\X]^E/\langle \Ps\rangle^E.$ 
 Let $F$ be the fraction field of $R[\X]^E/\langle \Ps\rangle^E$. Then we have that $\bar \alpha$ belongs to $\Ecl^{F}(R)$ (see Definition \ref{ecl}).
\qed 

\subsection{Real Nullstellensatz}
In this section we adapt our previous results to the case of $E$-fields (or $E$-rings) which are ordered (or admits an ordering, i.e. are orderable). We will refer to such $E$-fields as ordered $E$-fields (or ordered $E$-rings), respectively orderable ones.

\dfn \cite[Definition 2.1]{Lam} Let $(R,E)$ be an $E$-field. An ideal $I$ of $R$ is real if for any $u_1,\ldots,u_n\in R$ such that $\sum_{i=1}^n u_i^2\in I$, then $u_i\in I$ for all $i=1,\ldots,n$. 
A real $E$-ideal $I$ of $(R,E)$ is a real ideal which is also an $E$-ideal.
\par  We denote by $\Sigma$ the set of sums of squares in $R$.
\edfn

\ex  Let $R$ be an ordered $E$-domain, then the set of $E$-polynomials in $R[\X]^E$ which vanish on a subset of $R^n$ is a real $E$-ideal.
\eex

\lem  \label{aug-real} Let $B_{0}$ be a ring of characteristic $0$, $G$ a torsion-free abelian group and $B_1$ be the group ring $B_{0}[G]$. Let $I$ be a real ideal of $B_0$ and let $\phi^a$, $\phi^a_I$ as in Notation \ref{aug-nota}. The kernel $I_{1}$ in $B_1$ of the map $\phi^a_{I}$ is a real ideal of $B_{1}$. 
\elem
\pr Suppose that $\sum_{i=1}^n\; u_i^2\in I_1$, so $\sum_{i=1}^n\; \phi^a_I(u_i)^2=0$. Therefore, $\sum_{i=1}^n\; \phi^a(u_i)^2\in I$. Since $I$ is real, it implies that $\phi^a(u_i)\in I$ for all $1\leq i\leq n$, equivalently that $\phi^a_I(u_i)=0$ for all $1\leq i\leq n$.
\qed

\begin{lemma} \label{Lemma 8} Let $n\in\mathbb N$, $I_n\subseteq R_n$ be a  prime real ideal with the property $(\Ext)_n$. Then $I_n$ embeds in a  prime real ideal $I_{n+1}$ of $R_{n+1}$ with properties $(\Ext)_{n+1}$ and $(\Tr)_{n+1}$.
\end{lemma}
\begin{proof} By Lemma \ref{aug-real} and Proposition \ref{E}.
\end{proof} 

\begin{corollary}\label{Pos}  Let $(R,E,<)$ be an ordered $E$-field.
Let $f_1,\ldots, f_m\in R[\X]^E$ and let $k\in\mathbb N$ be chosen minimal such that $f_1,\cdots, f_m\in R_k$.  
Assume the ideal $I$ generated by $f_1,\cdots, f_m$ is disjoint from the set $1+\Sigma$. Let $\Ps$ be an ideal of $R_k$ maximal for the property of containing $I$ and being disjoint from  $1+\Sigma$. Suppose that $\Ps$ satisfies $(\Ext)_k$ 
Then $f_1,\cdots, f_m$ have a common zero in a partial ordered $E$-domain extending $(R,E)$. 
\ecor
\pr It is well-known that such ideal $\Ps$ is a prime real ideal  by \cite[Lemma 1.2, Remark 1.3, Definition 2.1]{Lam}.
It remains to apply Lemma \ref{Lemma 8} and Corollary \ref{E-E}, checking that a union of a chain of real ideals is a real ideal. So $\langle \Ps\rangle^E$ is a prime, real $E$-ideal. The quotient $D:=R[\X]^E/\langle\Ps\rangle^E$ is an orderable $E$-domain  \cite[Theorem 3.9]{Lam}. Moreover, $f_{1},\cdots,f_{m}$ have a common zero in $D$.
\qed 

\rem Note that when the exponential function satisfies the growth condition: $E(f)\geq 1+f$ for any $f\in R[\X]^E/\langle\Ps\rangle^E$, then by results of Dahn and Wolter \cite[Theorem 24]{DH}, one can embed this partial $E$-domain in a real-closed $E$-field where the exponential function is surjective on the positive elements.
\erem
 \medskip
 \par {\bf Acknowledgments:} These results are part of the PhD thesis of Nathalie Regnault \cite{R}. The second author would like to thank the model theory group of Manchester university for the opportunity to give a talk on this topic.

\end{document}